\documentclass[11pt,a4paper]{article}
\usepackage{natbib}

\usepackage{apalike}
\usepackage{comment}
\usepackage{lipsum}
\usepackage{microtype}

\usepackage[font=small,labelfont=bf]{caption}
\usepackage
[
        a4paper,
        left=3cm,
        right=3cm,
        top=3.5cm,
        bottom=3.5cm,
]
{geometry}

\usepackage{setspace}
\usepackage{graphicx}
\graphicspath{{./figures}}
\usepackage[skip=4pt plus1pt, indent=40pt]{parskip}
\usepackage{multicol,latexsym, array}
\usepackage{textcomp}
\usepackage[ansinew]{inputenc}
\usepackage{amsmath,amssymb, amsthm}
\usepackage[OT1]{fontenc}
\usepackage[english]{babel}
\usepackage{hyperref}
\hypersetup{breaklinks=true}

\hypersetup{
   colorlinks,
   linkcolor={magenta},
   citecolor={blue},
   urlcolor={blue}
}

\usepackage{enumitem}
\usepackage{MnSymbol} 
\usepackage{empheq}
\usepackage{rotating}
\usepackage{titletoc}
\usepackage{framed}
\usepackage{multirow,bigdelim}
\usepackage{subfigure}
\usepackage{bbm}
\usepackage[noend]{algpseudocode}
\usepackage{algorithm}
\usepackage{pgf}
\usepackage{pgfplots}
\usepackage{tikz}
\pgfplotsset{compat=1.17}
\usepackage{shadethm}
\usepackage{thmtools}
\usepackage{mdframed}
\usepackage{lipsum}
\usepackage{fixmath}
\usepackage{multicol}
\usepackage{footmisc}

\usepackage{caption}

\usepackage{xargs}  
\usepackage{xcolor} 

\definecolor{DarkGray}{RGB}{90,90,90}

\usepackage{cleveref}

\date{\today}

\newenvironment{ackno}%
    {
    \paragraph{Acknowledgements:} 
    }

\usepackage{caption}
\usepackage{mathrsfs}
\usetikzlibrary{arrows}

\usepackage{xpatch}

\makeatletter
\xpatchcmd{\endmdframed}
  {\aftergroup\endmdf@trivlist\color@endgroup}
  {\endmdf@trivlist\color@endgroup\@doendpe}
  {}{}
\makeatother

\usepackage{nicefrac}
\usepackage{dsfont}

\usepackage{bigints}
\usepackage{mathtools}

\usepackage{mathtools}
\usepackage{color}
\usepackage{fancyhdr}
\usepackage{lastpage}

\usepackage{ifthen}

\newcommand{\pms}{P}

\newcommand{\qms}{\nu}
\newcommand{\ems}{P_n}

\newcommand{\pmean}{\bar \pms}

\newcommand{\emean}{{\bar\pms}_n}
\newcommand{\tmean}{{\hat\pms}_n}

\newcommand{\RR}{\mathbb{R}}

\newcommand{\NN}{\mathbb{N}}

\newcommand{\Indicator}[1]{\mathds{1}_{#1}}

\newcommand{\XC}{\mathcal{X}}

\newcommand{\WC}{\mathcal{W}_2}
\newcommand{\SSS}{\mathbb{S}}
\newcommand{\EE}{\mathbb{E}}
\newcommand{\grad}{\mathrm{grad}}

\newcommand{\NC}{\mathcal{N}}

\newcommand{\PC}{\mathcal{P}}
\newcommand{\PF}{\mathfrak{P}}

\newcommand{\mf}{\mathbb{M}}

\newcommand{\Id}{\mathrm{Id}}

\newcommand{\TV}{\mathrm{TV}}
\newcommand{\norm}[1]{\left\lVert#1\right\rVert}

\newcommand{\dif}{\mathrm{d}}

\newcommand{\supp}{\mathrm{supp}}

\newcommand{\cov}{\,\mbox{\normalshape cov}}
\newcommand{\Cov}{\operatorname{Cov}}

\newcommand{\sign}{\mathrm{sign}}
\newcommand{\iid}{\stackrel{\!\mathrm{i.i.d.}}{\sim}\!}

\newcommand{\coloneqq}{:=}
\newcommand{\eqqcolon}{=:}

\renewcommand{\phi}{\varphi}

\newcommand{\diam}[1]{\mathrm{diam}(#1)}

\newcommand{\konvD}{\xrightarrow{\;\;\mathcal{D}\;\;}}

\renewcommand*{\epsilon}{\varepsilon}

\DeclareMathOperator*{\argmin}{\mathrm{argmin}}

\theoremstyle{plain}
\newtheorem{theorem}{Theorem}[section]
\newtheorem{corollary}[theorem]{Corollary}
\newtheorem{lemma}[theorem]{Lemma}
\newtheorem{proposition}[theorem]{Proposition}
\newtheorem*{theorem*}{Theorem}

\theoremstyle{definition}
\newtheorem{definition}[theorem]{Definition}
\newtheorem{remark}[theorem]{Remark}
\newtheorem{conjecture}[theorem]{Conjecture}
\newtheorem{example}[theorem]{Example}
\newtheorem*{example*}{Example}

\makeatletter
\def\mysequence#1{\expandafter\@mysequence\csname c@#1\endcsname}
\def\@mysequence#1{%
  \ifcase#1\or (NU)\or (B)\or (C)\or (D)\else\@ctrerr\fi}
\makeatother

\newtheorem{assumption}{Assumption}

\makeatletter
\newcommand{\mylabel}[2]{#2\def\@currentlabel{#2}\label{#1}}
\makeatother

\numberwithin{equation}{section}

\newcommand{\footremember}[2]{
	\footnote{#2}
	\newcounter{#1}
	\setcounter{#1}{\value{footnote}}
}

\newcommand{\footrecall}[1]{
	\footnotemark[\value{#1}]
}

\newif\ifexclude
\excludefalse  

\begin{document}

\title{A Lower Bound for Estimating Fr\'echet Means}
\author{Shayan Hundrieser
  \hspace{-0.55em}\footremember{ims}{\scriptsize
    Institute for Mathematical
		Stochastics, University of G\"ottingen,
		Goldschmidtstra{\ss}e 7, 37077 G\"ottingen}\\
  \footnotesize{\href{mailto:s.hundrieser@math.uni-goettingen.de}{s.hundrieser@math.uni-goettingen.de}}
  \\[2ex]
	Benjamin Eltzner
  \hspace{-0.25em}\footnote{\scriptsize
     Max Planck Institute for Multidisciplinary Sciences,
     Am Fa{\ss}berg 11, 37077 G\"ottingen}
	\\
  \footnotesize{\href{mailto:benjamin.eltzner@mpinat.mpg.de}{benjamin.eltzner@mpinat.mpg.de}
  }  \\[2ex]
	Stephan F. Huckemann
  \hspace{-0.6em}\footrecall{ims}\\
  \footnotesize{\href{mailto:huckeman@math.uni-goettingen.de}{huckeman@math.uni-goettingen.de}
  }
}
 
\pagenumbering{arabic}
\maketitle

\begin{abstract}
    \noindent 
    Fr\'echet means, conceptually appealing, generalize the Euclidean expectation to general metric spaces. We explore how well Fr\'echet means can be estimated from independent and identically distributed samples and uncover a fundamental limitation: In the vicinity of a probability distribution $\pms$ with nonunique means, independent of sample size, it is not possible to uniformly estimate Fr\'echet means below a  precision determined by the diameter of the set of Fr\'echet means of $\pms$. Implications were previously identified for empirical plug-in estimators as part of the phenomenon \emph{finite sample smeariness}. Our findings thus confirm inevitable statistical challenges in the estimation of Fr\'echet means on metric spaces for which there exist distributions with nonunique means. Illustrating the relevance of our lower bound, examples of extrinsic, intrinsic, Procrustes, diffusion and Wasserstein means showcase either deteriorating constants or slow convergence rates of empirical Fr\'echet means for samples near the regime of nonunique means.

    \end{abstract}
    \vspace{0.5cm}

    \vspace{0.5cm}
    
    \hyphenation{Hausdorff}

\section{Introduction}

 With the dawn of information age, richness and complexity of data have soared to unprecedented levels, emphasizing the relevance of taking into  account the underlying geometry of data.
 Whereas earlier methods of data analysis implicitly assumed data to be arranged in (flat) Euclidean spaces, recent techniques acknowledge and incorporate (non-flat) non-Euclidean aspects of data.
   This paradigm shift has lead to the development of a plethora of statistical methods for non-Euclidean data  which have proven instrumental in a multitude of disciplines, including directional statistics \citep{mardia2000directional}, shape analysis \citep{B91,kendall1989survey, kendall2009shape, small2012statistical, dryden2016statistical}, medical imaging \citep{arsigny2006log,pennec2006riemannian}, and phylogenetic tree analysis \citep{billera2001geometry, huson2006application, garba2021information}, to name a few. \cite{marron2021object} give a recent account of theory and application.
   
   A key tool to summarize the central tendency of a dataset involves location descriptors. In classical settings where data is assumed to reside on $\RR^m$, their expected value or their sample average (i.e., the Euclidean sample mean) are the most commonly used measure of centrality. 
   Generalizing this concept to a metric spaces $(\XC, d)$, \cite{frechet1948elements} proposed to consider for a random variable $X\sim \pms$ minimizers of the expected squared distance,\begin{align*}
\mf(\pms) = \argmin_{x\in \XC} \EE_{X\sim \pms}[d^2(x,X)].
\end{align*}
The, possibly empty, collection of such minimizers is termed the \emph{Fr\'echet mean (set)}. If it comprises a single element $\mf(\pms) =\{\pmean\}$, we speak  of uniqueness of the Fr\'echet mean. Moreover, in case $(\XC, d)$ is a Riemannian manifold equipped with geodesic distance this notion is referred to as \emph{intrinsic mean}, whereas if $(\XC, d)$ is embedded as a submanifold of $\RR^m$ and equipped with the Euclidean distance, the respective mean is also known as \emph{extrinsic means}.

In applications, the measure $P$ is typically unknown and instead only independent and identically distributed (i.i.d.) samples $X_1, \dots, X_n \sim \pms$ are available. This gives rise to the empirical measure $\ems \coloneqq \frac{1}{n} \sum_{i = 1}^{n} \delta_{X_i}$ and the associated \emph{empirical Fr\'echet mean (set)}, \begin{align*}
    \mf(\ems) =   \argmin_{x\in \XC} \frac{1}{n} \sum_{i =1}^{n} d^2(x,X_i). 
\end{align*}
A first observation about these empirical estimators is that under mild assumptions on the metric space and $P$ they are almost surely consistent as the sample size $n$ tends to infinity  \citep{ziezold1977expected, bhattacharya2003large, huckemann2011intrinsic,schotz2022strong, evans2024limit}. Building up on this insight, various statistical questions surrounding the performance and uncertainty of these estimators for increasing sample size arise.
 
Most work devoted to this subject relies on a critical assumption: uniqueness of the Fr\'echet mean for the probability distribution $\pms$. This topic has been the subject of a long line of research. 
A first set of contributions operates under the assumption of $\XC$ being a Riemannian manifold equipped with a geodesic metric and under considerable concentration of measures  \citep{karcher1977riemannian, kendall1990probability, le2001locating, groisser2005convergence,afsari2011riemannian}. Broadly speaking, under such assumptions the minimization problem for the Fr\'echet mean closely mimics that for the Euclidean mean, which are, if existent, always unique. Another regime imposes that the curvature of the space is non-positive \citep{sturm2003probability, bhattacharya2003large, afsari2011riemannian}. More recently, uniqueness of means has also been confirmed without imposing concentration conditions on the measures, e.g., on the circle \citep{hotz2015intrinsic}. Moreover, for empirical measures based on i.i.d.\ observations from a continuous distribution on a manifold, a.s.\ uniqueness has been confirmed by \cite{arnaudon2014means}. Beyond these settings, non-uniqueness of Fr\'echet means rather appears to be an exceptional case since a slight perturbation of the measure $\pms$ suffices to ensure uniqueness (\Cref{lem:instability}). 

Upon imposing uniqueness of Fr\'echet means, several articles contributed to the statistical analysis of Fr\'echet sample means. While initial works focused on distributional limits \citep{bhattacharya2005large, huckemann2011inference, bhattacharya2012nonparametric, barden2013central, bhattacharya2017omnibus, eltzner2021stability}, more recently there has been growing interest in convergence rates \citep{schotz2019arbitrary,schotz2019convergence, ahidar2020convergence,le2022fast} and concentration properties \citep{romon2023convex,escande2023concentration}. Most of these results impose concentration assumptions, as alluded to above, of the measures in conjunction with curvature conditions on the behavior of the Fr\'echet function near its minimum, and confirm a parametric $\mathcal{O}(n^{-1/2})$ convergence rate. This perspective has expanded over time, with more articles exploring the statistical behavior of empirical Fr\'echet means for distributions which are supported on larger domains \citep{mckilliam2012direction, hotz2015intrinsic,pennec2019curvature, eltzer2019_smearyCLT, eltzner2022geometrical}. A notable finding in this context has been the phenomenon of \emph{smeariness}, which asserts slower than parametric rates of convergence for empirical Fr\'echet means. 

Although smeariness might be interpreted as an anomaly, afflicting merely boundary cases of distributions, its repercussions are nonetheless significant, impairing the efficacy of empirical plug-in estimators for intrinsic means across practical data sets \citep{tran2021smeariness}. Such deterioration in performance spurred \cite{hundrieser2020finite} to devise a diagnostic tool aimed at quantifying this loss in performance of  Fr\'echet sample means from non-smeary distributions in finite sample regimes. Upon assuming uniqueness of Fr\'echet population and sample means, $\mf(\pms) = \{\pmean\}$ and $\mf(\ems)= \{\emean\}$ for $\pmean, \emean \in \XC$, their analysis relies on the behavior of the \emph{variance modulation} \citep{pennec2019curvature} of the Fr\'echet sample mean  as a function of the sample size~$n$, 
\begin{align}\label{eq:varMod}
    \mathfrak{m}_n:=\frac{n \EE[d^2(\emean, \pmean)]}{\EE[d^2(X, \pmean)]}
\end{align}
Intuitively, if this quantity is large, the performance of the empirical mean is much worse compared to what would be expected from the Euclidean setting (where $\mathfrak{m}_n\equiv 1$) and mimics the finite sample behavior from a sample mean of a smeary distribution. For this reason, estimators with a large variance modulation are said to be affected by \emph{finite sample smeariness}. 
This effect is particularly dominant over empirical intrinsic means in data from circular or spherical distributions with (nearly) full support \citep{hundrieser2020finite,eltzner2021finite} and extends to broader Riemannian manifolds that exhibit positive curvature \citep{pennec2019curvature}. While these findings were mainly developed for intrinsic means on manifolds, alternative notions of means, such as extrinsic means or the recently proposed diffusion means \citep{hansen2021diffusion, eltzner2023diffusion}, seem less susceptible to non-standard rates akin to smeariness, but nevertheless can still exhibit a pronounced deterioration in the constant of the convergence rate.

\vspace{-0.2cm}
\paragraph*{Contribution of this work.}
This loss in statistical efficacy raises the question: Is it possible to construct estimators for Fr\'echet means which are not affected by finite sample effects of smeariness? The central finding of this paper (\Cref{thm:minimaxLower}) lies in uncovering a fundamental limitation: We show that uniformly in the vicinity of distributions with a non-unique Fr\'echet mean, independent of the sample size, the estimation error for Fr\'echet means cannot be improved below a certain precision. This inherent lower bound cannot be overcome by whatever procedure estimating Fr\'echet means, with striking consequences for testing based on Fr\'echet means (\Cref{rmk:testing}). Revisiting prior work and establishing novel insights into extrinsic and intrinsic means, Procrustes means, diffusion means, and Wasserstein barycenters, we exemplarily illustrate the impact of the lower bound manifesting in deteriorating constants or slow convergence rates nearby the regime of population measures with non-unique means.

\vspace{-0.2cm}
\paragraph*{Outline.}
The work is structured as follows. Section \ref{sec:main} contains the main results of this work. Section \ref{sec:Implications} details the impact on various notions of means where the more involved arguments are deferred to the appendix. Finally, Section \ref{scn:discussion} discusses methods to detect and amend for this nearby  non-unique means. 

\vspace{-0.2cm}
\paragraph*{Notation.}
Throughout this work, $(\XC,d)$ denotes a metric space with its power set $\PF(\XC)$. Given a nonempty set $A\subseteq \XC$  its diameter is $\diam{A} \coloneqq \sup_{x,y\in A} d(x,y)$. The collection of Borel probability measures on $\XC$ is $\PC(\XC)$, and the total variation distance between two measures $\pms, Q\in\PC(\XC)$ is denoted by $\TV(\pms, Q)$. Further, if the (generalized) Fr\'echet mean of $P$ exists and is unique, we write $\pmean$ for the unique element in $\mf(P)$.

\section{Main Result}\label{sec:main}

 The central finding of this work is stated for the subsequent general notion of a mean. 

\begin{definition}[Fr\'echet $\rho$-means] \label{def:F-mean} Let $\rho\colon \XC\times \XC \to \RR$ be a Borel measurable function and let $\pms\in \PC(\XC)$ be a probability measure. Assuming $\rho(x,\cdot) \in L^1(\pms)$ for each $x\in \XC$, its \emph{Fr\'echet $\rho$-function} is defined as $$F_\pms\colon \XC \to \RR,\quad  x\mapsto \mathbb{E}_{X\sim \pms}[\rho(x,X)].$$ Further, its \emph{Fr\'echet $\rho$-mean (set)} is defined as the collection of minimizers $$\mf(\pms) \coloneqq \argmin_{x\in \XC} F_\pms(x)\in \PF(\XC).$$ Moreover, the Fr\'echet $\rho$-mean is said to be \emph{honest} if for every $x \in \XC$ the function $\rho(\cdot,x) \colon \XC\to \RR$ is uniquely minimized by $x$. 
\end{definition}

\begin{remark}[Honesty]
That $\mf$ is honest is equivalent to $\mf(\delta_x) = \{x\}$ for every $x \in \XC$. This means that $\mf$ always recovers the position of a Dirac measure at a single point and does not exhibit a systematic bias under strong concentration. 
\end{remark}

\begin{remark}[Relation to  standard Fr\'echet mean]
    The standard Fr\'echet mean arises by taking $\rho\coloneqq d^2$ where $d$ is the metric on $\XC$. In particular, Fr\'echet means are honest. 
\end{remark}

The framework of Fr\'echet $\rho$-means captures a large selection of notions of means on geometric structures, see \Cref{sec:Implications} for a number of examples, and extends the standard Euclidean mean. The latter would arise by choosing $\XC = \RR^m$ and $\rho(x,y) = \norm{x-y}^2-\norm{y}^2$. For this setting, it is known that a unique minimizer exists if the measure $\pms$ admits a finite (first) moment (see \cite{sturm2003probability}). Otherwise, the associated Fr\'echet function may not admit finite values and the Euclidean mean is empty. 

For a general metric space $\XC$ and a given measurable function $\rho$, the Fr\'echet $\rho$-mean also does not necessarily exist, and moreover, if it does, it may not be unique. Our main results assess effects induced by non-uniqueness, which we formalize by the following assumption.

\begin{assumption} 
\label{ass:nonUnique}
The Fr\'echet $\rho$-mean $\mf$ is honest and $\pms\in\PC(\XC)$ is a probability measure such that $\mf(\pms)$ is nonempty and non-unique, i.e., $|\mf(\pms)|>1$. 
\end{assumption}

The following result confirms that the condition of non-uniqueness of Fr\'echet $\rho$-means is a property, unstable under perturbations of measures. 

\begin{lemma}[Instability of non-uniqueness]\label{lem:instability}
   Under \Cref{ass:nonUnique}, it follows for all $\pms\in\PC(\XC)$, all $t\in (0,1]$ and all $x\in \mf(P)$ that
   $$\mf(\pms_{x,t}) = \{x\}$$
   with the perturbed measure $\pms_{x,t} \coloneqq (1-t)\pms + t \delta_x$.
\end{lemma}

\begin{proof}
    By \Cref{ass:nonUnique}, the Fr\'echet $\rho$-function fulfills for $\tilde x\in \XC\backslash\{x\}$ that
    \begin{align*}
        F_{\pms_{x,t}}(\tilde x) &=  (1-t) F_{\pms}(\tilde x) + t \rho(\tilde x,x)\\
        &>  (1-t) \min_{x'\in \XC}F_{\pms}(x') + t \min_{x''\in \XC}\rho(x'',x) \\
        &= (1-t) F_\pms(x) + t \rho(x,x) =  F_{\tilde \pms_{x,t}}(x),
    \end{align*}
    where the strict inequality is due to $\mf$ being honest. 
\end{proof}

This suggests non-uniqueness of Fr\'echet $\rho$-means to be rather exceptional. Nonetheless, in the immediate vicinity of a distribution with a non-unique mean, a significant loss in the performance of estimating population mean based on i.i.d.\ observations manifests.

\begin{theorem}[Uniform lower bound]\label{thm:minimaxLower}
	For given sample size $n\in \NN$, under \Cref{ass:nonUnique} it follows for all $\pms\in\PC(\XC)$, all $\epsilon \in (0,1)$ and all $p\geq 1$ that 
    \begin{align}\label{eq:minimaxLowerExp}
        \adjustlimits \inf_{\substack{\tmean\\[0.1cm] \phantom{.}}} \sup_{\substack{0 < t \leq \epsilon\\ x \in \mf(\pms)}} \mathbb{E}_{\pms_{x,t}}\left[d^p(\tmean(X_1, \dots, X_n), \pmean_{x,t})\right] \geq  \frac{1}{2}\left(\frac{\diam{\mf(\pms)}}{2}\right)^p,
\end{align}
    where the infimum is taken over all (measurable) estimators $\tmean \colon \XC^n \to \XC$ based on i.i.d.\ random variables $X_1, \dots, X_n\sim \pms_{x,t}\coloneqq (1-t)\pms + t \delta_x$. 
     Further, for all $\eta >0$ arbitrarily small it holds that 
     \begin{align}\label{eq:minimaxLower}\adjustlimits \inf_{\substack{\tmean\\[0.1cm] \phantom{.}}} \sup_{\substack{0 < t \leq \epsilon\\ x \in \mf(\pms)}} \mathbb{P}_{\pms_{x,t}}\left( d(\tmean(X_1, \dots, X_n), \pmean_{x,t}) \geq \frac{\diam{\mf(\pms)}- \eta}{2}\right) \geq  \frac{1}{2}.
     \end{align}
\end{theorem}

A key insight from \Cref{thm:minimaxLower} is that the uniform error in estimating the Fr\'echet $\rho$-mean does not improve as the sample size increases. This implies that under \Cref{ass:nonUnique} \emph{every} method of estimating the respective mean based on i.i.d.\ measurements will be affected by deteriorating constants or slower convergence rates (see \Cref{sec:Implications} for examples of both manifesting for various notions of means).  More precisely, for every estimation procedure there is a probability distribution for which it 
does not realize a  statistical accuracy below a precision determined by $\diam{\mf(\pms)}$: Choosing $p=2$ and noting that $1/9 < 1/8$ the assertion of the corollary below follows at once from (\ref{eq:minimaxLowerExp}) above. 

\begin{corollary}\label{cor:effects_FSS}
	Under \Cref{ass:nonUnique} for every estimator $\hat {\pms}\colon \bigcup_{n \in \NN} \XC^n \to  \XC$ and every sample size $n\in \NN$ there exists a probability measure $\pms_{\hat{\pms},n}\in \PC(\XC)$ for which the Fr\'echet $\rho$-mean $\mf(\pms_{\hat{\pms},n})=\{\pmean_{\hat{\pms},n}\}$ is unique such that for independent $X_1, \dots, X_n \sim \pms_{\hat{\pms},n}$ it follows
	\begin{align*}
  		 \mathbb{E}_{\pms_{\hat{\pms},n}}\left[ d^2(\hat \pms(X_1, \dots, X_n),\pmean_{\hat{\pms},n})\right] \geq  \diam{\mf(\pms)}^2/9. 
	\end{align*}
\end{corollary}

\begin{remark}[Comparison with \citet{tran2021smeariness}]
    The work by \cite{tran2021smeariness} asserts that nearby a smeary probability distribution, i.e., a distribution for which the empirical Fr\'echet means admit slower than parametric convergence rates, there exists a probability measure for which the variance modulation \eqref{eq:varMod}  of the empirical Fr\'echet mean can be arbitrary large for some sample sizes but stays bounded. This phenomenon has been called \emph{finite sample smeariness} by \cite{hundrieser2020finite}. In addition, \Cref{cor:effects_FSS} asserts that nearby the regime of non-uniqueness this phenomenon cannot be overcome by employing another estimation technique. 
\end{remark}

\begin{remark}[Complementary contributions]
    Previous contributions by \cite{schotz2019convergence, ahidar2020convergence, le2022fast} on the convergence behavior and by \cite{escande2023concentration,romon2023convex} on  concentration properties of empirical Fr\'echet means toward its population counterpart impose certain strictly positive curvature conditions on the behavior of the Fr\'echet function near its minimum. In particular, such assumption ensures that the Fr\'echet population mean is unique. Our work is concerned with the complementary regime where non-uniqueness occurs and establishes a lower bound which does not improve as the sample size increases. Insofar, our work confirms the necessity of separability conditions to infer uniform convergence results. 
    \end{remark}

\begin{proof}[Proof of \Cref{thm:minimaxLower}]
    We first establish the lower bound \eqref{eq:minimaxLower} in probability for which we use LeCam's two-point method \citep[Section 2]{tsybakov2009}. To this end, note by \Cref{ass:nonUnique} that $\diam{\mf(\pms)}>0$ and take $x,y\in \mf(\pms)$ such that $d(x,y) \geq \diam{\mf(\pms)}- \eta$.
    Then, for each $t \in (0,\epsilon]$ it follows by \Cref{lem:instability} that $\pmean_{x,t} = x$ and   $\pmean_{y,t} = y$, and thus
    \begin{align*}
        d(\pmean_{x,t}, \pmean_{y,t}) \geq \diam{\mf(\pms)} - \eta.
    \end{align*}
    Meanwhile, the total variation distance between $\tilde \pms_{x,t}$ and $\tilde \pms_{y,t}$ is bounded for each $t\in (0,\epsilon]$ by $\TV(\tilde \pms_{x,t}, \tilde \pms_{y,t}) \leq 2t.$ 
    Consequently, for $t^*\!\in (0,\epsilon]$ to be specified later, it follows from Inequality (2.9) of \citet{tsybakov2009} that the left-hand side of \eqref{eq:minimaxLower} is lower bounded by
    \begin{align}
        &\adjustlimits \inf_{\tmean} \max_{z \in\{x,y\}} \mathbb{P}_{\pms_{z,t^*}}\left( d(\tmean(X_1, \dots, X_n), \pmean_{z,t^*}) \geq \frac{\diam{\mf(\pms)}- \eta}{2}\right)\notag\\
       \geq \;& \adjustlimits \inf_{\hat \Psi_n} \max_{z \in\{x,y\}} \mathbb{P}_{\pms_{z,t^*}}\left( \hat\Psi_n(X_1, \dots, X_n) \neq z\right), \label{eq:lowerbound_test}
    \end{align}
    where the infimum in the second line is taken over all (measurable) tests $\hat \Psi_n\colon \XC^n \to \{x,y\}$.
    Using Theorem 2.2(i) and LeCam's inequalities (Lemma 2.3) in conjunction with the tensorization property of the Hellinger distance (p.\ 83, Property (iv)) of \cite{tsybakov2009} it follows that the right-hand side of \eqref{eq:lowerbound_test} is lower bounded by 
    \begin{align}\label{eq:lowerbound_TV}
        \frac{1}{2}\left(1 - \TV(\tilde \pms_{x,t^*}^{\otimes n}, \tilde \pms_{y,t^*}^{\otimes n})\right) \geq \frac{ 1- \sqrt{1 - (1-2t^*)^n}}{2}.
    \end{align}
    Hence, by choosing $t^*>0$ arbitrarily small, we conclude the validity of inequality \eqref{eq:minimaxLower}.

 For the lower bound \eqref{eq:minimaxLowerExp} in expectation we use Markov's inequality with \eqref{eq:minimaxLower}. This yields for $p \geq 1$ and $\eta \in (0,\diam{\mf(\pms)})$  the left-hand side of \eqref{eq:minimaxLowerExp} to be lower bounded~by
    \begin{align*}
        & \left(\frac{\diam{\mf(\pms)}- \eta}{2}\right)^p \adjustlimits \inf_{\substack{\tmean\\[0.1cm] \phantom{.}}} \sup_{\substack{0 < t \leq \epsilon\\ x \in \mf(\pms)}} \mathbb{P}_{\pms_{x,t}}\left( d(\tmean(X_1, \dots, X_n), \pmean_{x,t}) \geq \frac{\diam{\mf(\pms)}- \eta}{2}\right) \\
        \geq \;  & \left(\frac{\diam{\mf(\pms)}- \eta}{2}\right)^p   \frac{1}{2}.
    \end{align*}
    Choosing $\eta>0$ arbitrarily small yields the assertion. 
\end{proof}

\begin{remark}[Impact on testing]\label{rmk:testing}
    The proof of \Cref{thm:minimaxLower} (Inequalities \eqref{eq:lowerbound_test} and \eqref{eq:lowerbound_TV}) also reveals that under \Cref{ass:nonUnique} every test $\hat\Psi_n \colon \XC^N \to \{x,y\}$, to assess if the Fr\'echet mean is located at $x$ or $y$,  experiences a considerable loss in performance if the data stems from a probability distribution that is close to $P$ with $x,y \in \mf(\pms)$. In fact, uniformly over the collection $\{\pms_{z,t} \,|\, z \in \{x,y\}, t\in (0,\epsilon)\}$ no test  performs better than a random coin flip. 
\end{remark}

\section{Slow Rates and Deteriorating Constants: Examples}\label{sec:Implications}

The previous section highlighted inherent statistical challenges in estimating Fr\'echet means nearby the regime of non-uniqueness. 
For illustration, in this section, we revisit and expand beyond the literature specific examples of Fr\'echet means, demonstrating \emph{slow rates} or \emph{deteriorating constants} in the convergence behavior of empirical plug-in estimators. 
Throughout these examples, we always consider i.i.d.\ observations $X_1, \dots, X_n \sim \pms$ from a probability distribution $\pms$, which serve to define the empirical measure $\ems\coloneqq \frac{1}{n}\sum_{i = 1}^{n} \delta_{X_i}$.  

Given a measurable and integrable function $\rho$ as in \Cref{def:F-mean}, this leads to a notion $\mf$ of a Fr\'echet $\rho$-mean with population (resp. empirical) mean 
$\mf(\pms)$ (resp. $\mf(\ems)$) and 
population (resp. empirical) Fr\'echet functions 
$F_\pms$ (resp. $F_n \coloneqq F_{\ems}$). 

\subsection{Extrinsic Means on Manifolds}\label{subsec:extrinsicMeans}

Let $\XC$ be a smooth manifold (e.g. \citealt{lee2010introduction} for a textbook) that is smoothly embedded via $j \colon \XC\to \RR^m$ in the Euclidean space and equipped with Euclidean distance $d_E\colon \XC \times \XC \to [0, \infty), (x,y) \mapsto \norm{j(x)-j(y)}$. Then, the \emph{extrinsic mean} set $\mf_E$ is defined as the Fr\'echet mean set on $\XC$ with respect to $\rho = d_E^2$. 

Given a probability distribution $\pms\in \PC(\XC)$, \cite{HL96,HL98} obtained  its \emph{mean location}, which \cite{bhattacharya2003large} called its \emph{extrinsic mean}, thus by minimizing the Euclidean distance under constraining to the manifold. Hence, the extrinsic mean set is given by the orthogonal projection of the Euclidean mean $\EE_{j_{\#}\pms}[X]$ in ambient space (of the pushforward measure $j_{\#}\pms$) to the embedded manifold
\begin{align*}
  \mf_E(P) = \argmin_{x \in \XC} \norm{j(x) - \EE_{j_{\#}\pms}[X]}\,.
\end{align*}
If nonvoid (e.g. if the manifold is compact), uniqueness of $\mf_E$ therefore occurs if and only if $\EE_{j_{\#}\pms}[X]$ is located at a position $y\in \RR^m$ for which there exists a unique orthogonal projection $x \in \XC$, i.e., $\norm{j(x) - y}< \norm{j(x') - y}$ for all $x' \in \XC$. The collection of all such points is termed \emph{non-focal points}, whereas elements of its  complement are called \emph{focal points}.

In consequence, \Cref{thm:minimaxLower} finds application for a distribution $\pms$ if the Euclidean mean of the embedded distribution is a focal point. 

A simple example arises in the context of spheres $\SSS^m\coloneqq \{x\in \RR^{m+1} \,|\, \norm{x}_2 = 1\}$ for $m \geq 1$ which can be embedded into $\RR^{m+1}$ via the inclusion map. Here, every point $\RR^m\backslash\{0\}$ is non-focal. Since the distance minimizing projection is given by the normalization $x\mapsto x/\norm{x}$, we obtain the following characterization of extrinsic mean sets on the sphere,\begin{align*}
  \mf_E \colon \PC(\SSS^m) \mapsto \PF(\SSS^m), \quad \pms \mapsto \begin{cases}
 	\{ \EE_{\pms}[X] / \norm{\EE_{\pms}[X]} \} & \text{ if } \EE_{\pms}[X] \neq 0,\\
 	\; \SSS^m & \text{ if } \EE_{\pms}[X] = 0. 
 \end{cases}
\end{align*}
Hence, by \Cref{thm:minimaxLower}, the empirical extrinsic mean will likely perform statistically worse if the Euclidean mean $\EE_{\pms}[X]$ is close to the origin. Indeed, since the orthogonal projection to the tangent space of the sphere at $x\in \SSS^m$ is explicitly given by
$$ \RR^m \to \{z \in \RR^m: z^Tx=0\}, y \mapsto  (\Id_{n+1}- xx^T)y\,$$
the \emph{deteriorating constants} can be directly observed in the following, see also \cite{bhattacharya2005large}.

\begin{theorem}[\citealt{HLR96}, CLT extrinsic spherical mean] \label{prop:CLT-extrinsic-sphere}
	Let $\pms\in \PC(\SSS^m)$ be a probability distribution such that $\EE_{\pms}[X]\neq 0$. Then, the extrinsic mean $\mf_E(\pms)= \{\pmean\}$ is unique and for i.i.d.\ random elements $X_1, \dots, X_n \sim \pms$ it follows, as $n\to \infty$, for a measurable selection of extrinsic sample means $\emean \in  \mf_E(\ems)$ that 
	\begin{align*}
  \sqrt{n}( \emean - \pmean ) \konvD \NC(0, C \Sigma C^T).
	\end{align*}
Herein, the asymptotic covariance is determined by matrices $\Sigma = (\Cov[X_1^i,X_1^j])_{i,j= 1}^{m+1}$ and $C = (\Id_{m+1} - \pmean\pmean^T)/\norm{\EE_{\pms}[X]}$.
\end{theorem}

From this distributional limit it is apparent that the convergence rate of the empirical extrinsic spherical mean will always be of parametric order. However, the variance of the limit distribution is reciprocally linked to the proximity of the Euclidean mean $\EE_{\pms}[X]$ to the origin, an observation which aligns with \Cref{thm:minimaxLower}.

\subsection{Intrinsic Means on Manifolds}

A notable conceptual drawback of extrinsic means is that they depend on the embedding. 
The following alternative notion of a mean only depends on intrinsic properties: Let
$\XC$ be a 
Riemannian manifold (e.g. \citealt{lee2018introduction} for a textbook) 
with its intrinsic (also called geodesic) distance $d_I$ induced by the Riemannian metric. The intrinsic mean $\mf_I$ is then defined as the Fr\'echet mean on $\XC$ with respect to $\rho = d_I^2$. 

In the following we focus on two Riemannian manifolds with positive (global) curvature, the circle and the sphere, which feature probability distributions with nonunique means (e.g.\ the uniform distribution) 
and where the implications of \Cref{thm:minimaxLower} have been observed for distributional limits via slower than parametric convergence rates or increasing asymptotic variances.

\paragraph*{The circle } is 
the space $\SSS = [-\pi,\pi)$ with $-\pi$ and $\pi$ identified and equipped with the usual (intrinsic) arc length distance $d_{I}(x,y) = \min\{|x-y|, 2\pi -|x-y|\}.$  An extensive analysis of intrinsic circular means has been carried out by \cite{hotz2015intrinsic}. According to their findings, for a probability measure $\pms$ and a mean element $a\in \mf_I(\pms)$ the antipodal point $q(a) \coloneqq a +\pi\Indicator{[-\pi,0)}(a) - \pi\Indicator{[0,\pi)}(a)$ does not admit positive mass. Further, if $\pms$ features a continuous density near $q(a)$, then it is dominated by $\frac{1}{2\pi}$, the density of the uniform distribution \citep[Theorem 1]{hotz2015intrinsic}. Moreover, if the density is equal to $\frac{1}{2\pi}$ on a neighborhood near $q(a)$, then $a$ is not a unique intrinsic mean \cite[Corollary 1]{hotz2015intrinsic}. 

The following details and extends general central limit theorems from \cite{mckilliam2012direction, hotz2015intrinsic}, highlighting an intimate relation between the asymptotic law and the behavior of the underlying probability measure near the antipode of the intrinsic population mean.

\begin{proposition}[CLT intrinsic circular mean]\label{prop:CLT_Circle}
    Let $\pms \in \PC(\SSS)$ be a probability distribution with a unique intrinsic mean $\mf_I(\pms) = \{0\}$, assume that $\pms$ features a continuous density on $(\pi-\delta, \pi)\cup [-\pi, -\pi + \delta)$ with respect to the arc length measure for some $\delta>0$, and denote by $\sigma^2 \coloneqq \EE_{X\sim \pms}[X^2]$ the Euclidean variance. Further, for $n\in \NN$ consider i.i.d.\ random variables $X_1, \dots, X_n \sim \pms$ and a measurable selection of intrinsic sample means $\emean \in \hat \mf_I(\ems)$
	\begin{itemize}
	 \item[$(i)$] In case $f(-\pi) = \lim_{x \searrow -\pi} f(x) = \lim_{x \nearrow \pi} f(x) < \frac{1}{2\pi} $, it follows  as $n \to \infty$ that
	 $$    \sqrt{n} \,\emean  \xrightarrow{\mathcal{D}} \mathcal{N}\left(0,\sigma^2 \big(1-2\pi f(-\pi)\big)^{-2}  \right). $$
     \item[$(ii)$] In case $f(-\pi + \epsilon) = \frac{1}{2\pi} - G'(\epsilon) + o(G'(\epsilon))$ and $
    f(\pi - \epsilon)= \frac{1}{2\pi} -  G'(\epsilon) + o(G'(\epsilon))$, for a non-negative, locally at $0$ strictly convex, continuously differentiable function $G\colon [0, \delta) \rightarrow \RR_{\geq 0}$ fulfilling $G(0)=G'(0)=0$, it follows  as $n \to \infty$ that
    $$  \sqrt{n} \, \sign(\emean)\, 2\pi\, G(|\emean|) \xrightarrow{\mathcal{D}} \mathcal{N}\left(0,\sigma^2 \right)\,. $$
	\end{itemize}
    \end{proposition}	
	
    Similar to the distributional limit for empirical extrinsic means (\Cref{prop:CLT-extrinsic-sphere}), here the asymptotic variance increases as the antipodal density $f(-\pi)$ tends to $\frac{1}{2\pi}$. Further, if the density attains the value  $\frac{1}{2\pi}$ and is sufficiently regular near the antipode, the convergence rate can be slower than parametric, exemplified by the following two types of smeariness.

    \begin{example}[Power smeariness] \label{ex:power_smeary}
     For $r>0$ consider the probability measure, continuous with respect to arc length,
  \begin{align*} \dif\pms_{\textup{pow},r}(x) =& \frac{(\pi -1)r + \pi}{\pi r + \pi}\mathds{1}_{[-1/2,1/2]}(x) \dif x +  \frac{1}{2\pi}\cdot \Big(  1- \big(\pi - x\big)^r  \Big)\mathds{1}_{[\pi-1, \pi)}(x)\dif x \\
  &+ \frac{1}{2\pi}\cdot \Big(  1- \big(\pi+x\big)^r\Big)\mathds{1}_{[-\pi,-\pi+1]}(x)\dif x\,.  
\end{align*}
Symmetry and \citet[Proposition 1]{hotz2015intrinsic} yield a unique intrinsic population mean at $\mf_I=\{ 0\}$ with $G(\epsilon) = {\epsilon^{r+1}/(2\pi (r+1))}$ in $[0,1)$. Hence, by \Cref{prop:CLT_Circle}(ii), for intrinsic sample means $\emean$ based on $X_1, \dots, X_n \sim \pms_{\textup{pow},r}$, as $n\to \infty$, it follows that $$n^{1/(2r+2)}  \emean$$ has a nontrivial limiting distribution, so that $\pms_{\textup{pow},r}$ is \emph{power smeary of order} $r$. Notably, the special case of power smeariness with \emph{integer order} $r\in \{0,1,2,\ldots\}$ was also derived by \cite{hotz2015intrinsic}.
\end{example}

    \begin{example}[Logarithmic smeariness]\label{cor:log-smeary}
 For $r>0$ and suitable $c_r>0$ consider the probability measure, continuous with respect to arc length, 
  \begin{align*}
    \dif\pms_{\textup{log},r}(x) =& \;c_r\cdot \mathds{1}_{[-1/2, 1/2]}(x)\dif x + \mathds{1}_{[\pi-1/2,\pi)}(x)\,\frac{\max\left\{1 - G_r'(\pi-x) , 0\right\}}{2\pi}\,\dif x\\
    &+\mathds{1}_{[-\pi,-\pi+1/2]}(x)\,\frac{\max\left\{1 - G_r'(\pi +x ) ,0\right\}}{2\pi}\,\dif x\,,
  \end{align*}
 where $G_r(x) =  \exp(-1/|x|^r)\Indicator{(0,\infty)}(x)$ is a smooth function.
 Symmetry and \citet[Proposition 1]{hotz2015intrinsic} imply the intrinsic population mean to be unique $\mf_I=\{ 0\}$. Then, by the strong law of \cite{ziezold1977expected} and \Cref{prop:CLT_Circle}(ii) it follows (see ``Proof of Example \ref{cor:log-smeary}'' in the appendix for a complete argument) for the intrinsic sample means $\emean$ based on $X_1, \dots, X_n \sim \pms_{\textup{log},r}$, as $n\to \infty$, that we have a threefold limiting behavior
 \begin{enumerate}
    \item[$(i)$] $\emean\xrightarrow{\rm a.s.}  \;0$,
    \item[$(ii)$] $  \left(\log \sqrt{n}\right)^{1/r} \emean \xrightarrow{\;\;\mathcal{D}\;\ }Z$ with $Z \sim \frac{1}{2}( \delta_{-1} + \delta_{1})$,
    \item[$(iii)$] $   r\,\left(\log\sqrt{n}\right)^{(1+r)/r}  \Big( \emean -\frac{\sign{(\emean)}}{\left(\log\sqrt{n}\right)^{1/r}}  \Big) \xrightarrow{\;\;\mathcal{D}\;\ }\sign{(Z')}\cdot \log{|Z'|}$ with $Z' \sim \mathcal{N}\left(0,\frac{\EE[X_1^2]}{4\pi^2}\right) $.
   \end{enumerate}
\end{example}

\paragraph*{The Sphere} of dimension $m \geq 2$, recall \Cref{subsec:extrinsicMeans}, is now viewed as the Riemannian immersion $\SSS^m =\{x\in \RR^{m+1} \,|\, \norm{x}_2 = 1\}$, 
 inducing the (spherical) intrinsic metric $d_{I}(x,y) = \arccos^{-1}\langle x, y \rangle$ for $x,y \in \SSS^m$. Further, with the standard unit column vectors $e_1, \dots, e_{m+1}$ of $\RR^{m+1}$, $\xi = e_1$ denotes the \emph{north pole}, coming with the \emph{inverse exponential chart}
\begin{align*}
\exp_{\xi}^{-1} \colon \SSS^m\backslash\{-\xi\}\to \RR^{m}, \quad x \mapsto (e_2, \dots, e_{m+1})^T(x- \langle x,\xi\rangle \xi) \frac{\arccos\langle x,\xi\rangle }{\norm{x-\langle x,\xi\rangle \xi}}\,.
\end{align*}
The topic of uniqueness of intrinsic means on the multivariate sphere is not as thoroughly studied as for the univariate setting. An important observation in this context is that if the underlying probability measure $\pms$ is rotationally symmetric on $\SSS^m$, then by \citet[Example 2.2]{bhattacharya2003large} the intrinsic mean $\mf_I(\pms)$ will be the union of parallel $(m-1)$-dimensional spheres or poles of the rotation axis. 

An example of a family of distributions close to the regime of non-unique intrinsic means is detailed in Section 4 of \cite{eltzer2019_smearyCLT}. They considered the parametric family of probability distributions $\pms_\alpha \in\PC(\SSS^m)$ indexed by $\alpha \in (0,1)$, where $\pms_\alpha$ is the mixture of a uniform distribution on the lower half-sphere $\mathbb{L}^m = \{ x \in \SSS^m \,|\, x_1 \leq 0\}$ weighted by $\alpha$ and a Dirac measure at the north pole $\xi$ weighted by $(1-\alpha)$. Then, according to \citet[Lemma 4.1 and its proof]{eltzer2019_smearyCLT} it follows that the intrinsic mean is given by $\mf_I(\pms_\alpha) =\{\xi\}$ if and only if 
$$\alpha \leq  \frac{1}{1+\gamma_m} \quad \text{with} \quad \gamma_m\coloneqq \frac{\sqrt{\pi}}{2} \frac{\Gamma(\frac{m+1}{2})}{\Gamma(\frac{m+2}{2})},$$
where $\Gamma \colon (0, \infty) \to (0, \infty)$ denotes the Gamma-function. 
They obtain this characterization by explicitly computing the Fr\'echet population function. 
In particular, for $\alpha < 1/(1+\gamma_m)$ the Hessian of the Fr\'echet population function parametrized in the exponential chart at the north pole is positive definite, whereas at the boundary to non-uniqueness, $\alpha = 1/(1+\gamma_m)$, all second and third order derivatives vanish while the fourth-order derivative forms a diagonal fourth-order tensor with positive entries. The following result establishes the distributional limit for intrinsic sample means for the different regimes in $\alpha$; it is a consequence of \citet[Theorems 2.11 and 4.3]{eltzer2019_smearyCLT}.

\begin{example}[CLT intrinsic spherical mean]
       Take $\pms_\alpha \in\PC(\SSS^m)$ as above for $m\geq 2$ and $\alpha \geq 1/(1+\gamma_m)$. Further, for $n\in \NN$ consider i.i.d.\ random variables $X_1, \dots, X_n \sim \pms_\alpha$ and a measurable selection of intrinsic sample means $\emean  \in \hat \mf_I(\ems)$. 
	\begin{itemize}
	 \item[(i)]  In case $\alpha < 1/(1+\gamma_m)$, it follows  as $n \to \infty$ that
	 $$    \sqrt{n} \, \exp_\xi^{-1}(\emean)  \xrightarrow{\mathcal{D}} \mathcal{N}\left(0,\Sigma(\alpha) \right), $$
  where $\Sigma(\alpha) \in \RR^{m\times m}$ is a symmetric positive definite matrix with identical eigenvalues $\lambda_1(\alpha) = \dots = \lambda_m(\alpha)$ which  fulfill $\lambda_1(\alpha) \asymp (1-\alpha (1+ \gamma_m))^{-2}$ for $\alpha \nearrow 1/(1+ \gamma_m)$. 
     \item[(ii)] In case $\alpha = 1/(1+\gamma_m)$, it follows  as $n \to \infty$ that
    $$  n^{-1/6}\, \exp_\xi^{-1} (\emean) \xrightarrow{\mathcal{D}} \mathcal{H}\,, $$
    where $\mathcal{H} = (\mathcal{H}_1, \dots, \mathcal{H}_m)$ is a real-valued random variable such that $(\mathcal{H}_1^3, \dots, \mathcal{H}_m^3)$ is distributed according to centered multivariate normal with positive definite covariance. 
	\end{itemize}
\end{example}

The asymptotic fluctuation aligns with our main result (\Cref{thm:minimaxLower}) and confirms a significant (worsening of the variance) deterioration of constants as the parameter $\alpha$ approaches $1/(1+\gamma_m)$. In particular, at the boundary $\alpha = 1/(1+\gamma_m)$ a slower than parametric convergence rate manifests. In the language of  \Cref{ex:power_smeary} this corresponds to power smeariness of order $r = 2$. 

Notably, by Stirling's formula, it follows for $m \to \infty$ that $\gamma_m \to \infty$, which suggests that in high-dimensional settings issues related to non-uniqueness will occur with even a very small contribution of mass near the antipode of the north pole. This aligns with the well-known \emph{curse of dimensionality}, which suggests that estimation of parameters becomes increasingly challenging as the ambient dimension increases. Moreover, on high dimensional spheres, non-unique intrinsic population means and the associated deterioration in estimation capabilities is not just limited to assuming mass antipodal to the mean, but can also occur if the underlying probability measure is not concentrated in a half sphere. For a detailed account on this and resulting effects of smeariness we refer to \cite{eltzner2022geometrical}. 

\subsection{Procrustean Means for Shape Spaces} 

Procrustes means have been introduced by \cite{Gow} for landmark based shape spaces. The latter describe geometrical objects in $\RR^m$ determined by placing $m+1\leq k\in \NN$ landmarks at corresponding locations on each object, amounting to data matrices in $\RR^{m\times k}$, modulo the action of a suitable group. For the group of similarity transformations, accounting for translation amounts to data matrices in $\RR^{m\times (k-1)}$, accounting for size yields unit size data matrices in $\SSS^{m\times (k-1)-1} \subset \RR^{m\times (k-1)}$ and, accounting for common rotation of all landmarks by the group $SO(m)$ of rotations in $\RR^m$ yields the \emph{shape space}
$$\Sigma_m^k := \{[x]: x \in \SSS^{m\times (k-1)-1}\},\quad [x]= \{gx : g\in SO(m)\}$$
of equivalence classes, as detailed, e.g. in \cite{dryden2016statistical}. The \emph{Procrustes metric}
$$ d_P([x],[y]) := \min_{g\in SO(m)} d_R(x,gy),\quad x,y \in \SSS^{m\times (k-1)-1}\,,$$
usually considered, is the quotient metric from the residual quasi-metric
$$ d_R(x,y) = \sqrt{1 - {\rm tr}(x^Ty)^2},\quad x,y \in \SSS^{m\times (k-1)-1}\,,$$
(e.g. \cite{H_meansmeans_12}); it is symmetric, satisfies the triangle inequality but $d_R(x,-x)=0$. The corresponding Fr\'echet $\rho$-means $\mathbb M(P)$, $P\in \mathcal P(\Sigma_m^k)$, $\rho = d^2_P$, are called \emph{Procrustes means}. Means similarly constructed with respect to different group actions, are \emph{Procrustean}.

The planar shape spaces $\Sigma_2^k$ carry the canonical structure of complex projective spaces $\mathbb C P^{k-2}$, Riemannian manifolds of real dimension $2(k-2)$, inducing the Procrustes metric. For higher dimensions ($m>2$) the shape spaces are stratified spaces, nonmanifolds comprising manifolds of varying dimensions with highest dimensional manifold stratum $\Sigma^*$ which is open and dense in $\Sigma_m^k$, due to the \emph{principal orbit theorem} \citep[Chapter IV]{Bre72}. Notably, while the \emph{manifold stability theorem} (means of random shapes assuming $\Sigma^*$ with positive probability lie on $\Sigma^*$ themselves, see \cite[Corollary 1]{H_meansmeans_12}) does not hold for Procrustes means in general \citep[Section 4.4]{H_meansmeans_12},  it seems to hold for most applications. Moreover \citep[Theorem 3]{H_meansmeans_12}, for every $x \in [x]\in \Sigma^*$ there is a measurable \emph{lift} $\psi_x: \Sigma^* \to \SSS^{m\times (k-1)-1}$ \emph{in optimal position to $x$}, i.e.
$$[\psi_x([x'])] = [x']\mbox{ and }d_P([x],[x']) = d_R(x,\psi_x([x']))\mbox{ for all }[x'] \in \Sigma^* \,,$$
that is locally smooth near $[x]$. 
Hence, a random shape $Z\sim P \in {\mathcal P}(\Sigma_m^k)$ a.s. assuming $\Sigma^*$  can be lifted to $\SSS^{m\times (k-1)-1}$ in optimal position to $\xi \in \zeta \in \mathbb M(P)$, if $\mathbb M(P) \subset \Sigma^*$ -- this is the case for most realistic applications. 
Since for any other $\xi' \in \SSS^{m\times (k-1)-1}$ and random $X := \psi_{\xi}(Z)$ on $\SSS^{m\times (k-1)-1}$,
\begin{align*}
\EE[d_R(\xi',X)^2] &\geq \EE[d_R(\xi',\psi_{\xi'}([X]))^2] = \EE[d_P([\xi'],[X])^2] 
\\&\geq \EE[d_P([\xi],[X])^2] = \EE[d_R(\xi,X)^2]\,,
\end{align*}
$\xi$ is a Fr\'echet $\rho$-mean (with $\rho = d^2_R$) of $\psi_\xi\sharp P$, called a \emph{residual mean}.

Vectorizing via 
\begin{align*}
    \SSS^{m\times (k-1)-1} &\to  \SSS^{m(k-1)-1}\\ (x_1,\ldots,x_j) =x &\mapsto {\rm vec}(x) = (x_1^T,\ldots,x^T_k)^T\,,
\end{align*}
due to $\EE[d_R^2(x,X)] = 1- {\rm vec}(x)^T\EE[{\rm vec}(X){\rm vec}(X)^T]{\rm vec}(x)$,
the residual mean set of a random configuration matrix $X \in \SSS^{m\times (k-1)-1}$ is nonempty and given by the inverse vectorization of the intersection of $\SSS^{m\times (k-1)-1}$ with the eigenspace to the largest eigenvalue of $\EE[{\rm vec}{(X)\rm vec}(X)^T]$. 

Finally, with the tangent space $T_\xi \SSS^{m\times (k-1)-1} = \{y\in \RR^{m\times (k-1)}: {\rm tr}(y^T\xi) =0\}$ of $\SSS^{m\times (k-1)-1}$ at $\xi\in \SSS^{m\times (k-1)-1}$, introduce a local chart of $\SSS^{m\times (k-1)-1}$ near $\xi$ via
$$\phi_\xi : x \mapsto \frac{x - {\rm  tr}(x^T\xi) \xi}{{\rm  tr}(x^T\xi)}$$
for $x$ in the hemisphere determined by ${\rm  tr}(x^T\xi)>0$. Notably, if $\phi_\xi(x)=y$, then $x=\frac{y+\xi}{\|y+\xi\|}$.

\begin{proposition}[CLT Procrustes means]\label{prop:Procrustes} Let $Z_1,\ldots,Z_n, Z \iid  P\in {\mathcal P}(\Sigma_m^k)$ be random shapes, $2\leq m < k$, that a.s. assume $\Sigma^*$ and suppose that $\bar P \in\Sigma^*$ is the unique Procrustes mean of $P$. Further suppose that $\bar P_n$ is a measurable selection of the 
Procrustes mean of the empirical law $P_n = \frac{1}{n}\sum_{i=1}^n \delta_{Z_i}$. Then with $\xi \in \bar P$, the above local chart $\phi_{\xi}$ and the measurable lift $\psi_{\xi}$, smooth near $\xi$, from above, we have
$$ \sqrt{n}\, {\rm vec}\left(\phi_{\xi} \circ \psi_{\xi} (\bar P_n)\right) 
 \xrightarrow{\mathcal{D}} {\mathcal N}\left(0,\,H^{-}DH^{-}\right)$$
with $N=m(k-1)$,  $D=\cov[{\rm vec}(\psi_\xi(Z)){\rm vec}(\psi_\xi(Z))^T] = \sum_{J=1}^{N+1}\lambda_j v_j$, $\lambda_1 > \lambda_2 \geq \ldots \geq \lambda_{N+1}$, ${\rm vec}(\xi)=v_1,\ldots,v_{N+1} \in \SSS^N$, and 
$H^-= \frac{1}{2} \sum_{j=2}^{N+1}(\lambda_1 - \lambda_j)^{-1} v_jv_j^T$.
\end{proposition}

The proof is deferred to the appendix. 
Notably, as detailed in that proof, the columns of the generalized inverse  $H^-$ of the Hessian matrix span the tangent space of $\Sigma^*$ at $\bar P$.  When the Procrustes mean ceases to be unique, e.g. if $\lambda_2 \to \lambda_1$, due to symmetries in the distribution $P$, say, then the ``constants above deteriorate''.

\subsection{Diffusion Means on Manifolds} 

Notably in this section, $t$ denotes (diffusion) time, and not perturbation scale as in the rest of this paper.

Extending upon the idea of means on manifolds which shall be independent of the embedding, \cite{hansen2021diffusion} and \cite{eltzner2023diffusion} recently proposed a concept of a mean on geodesically and stochastically complete, connected Riemannian manifolds. The so-called \emph{diffusion mean} $\mf_{D,t}(P)$ for diffusion time $t$ are defined as maximum likelihood estimates for the heat kernel, the unique symmetric solution of the heat equation for a point mass as initial condition at time $t=0$. Diffusion means for time $t$ are location statistics that can be treated in the framework of Fr\'echet $\rho$-means by selecting $\rho_t(x,y)$ as the negative log-likelihood at point $y$ of a Brownian motion process at time $t$ started at $x$. For diffusion time $t\to 0$ the diffusion mean set converges to a subset of the intrinsic mean set. It is an open question whether there are examples where this subset is proper. For circles, spheres and real projective spaces the limit for $t\to \infty$ yields a subset of the extrinsic mean set, as recently shown by \cite{Duesberg2024thelong}, and, he provided examples where this subset is proper.

On $\mathbb{R}^m$, $\SSS^m$ and hyperbolic spaces $\mathbb{H}^m$ (see \cite{lee2018introduction}), diffusion means are honest, since in these cases the log-likelihood $\rho_t(x,y)$ can be expressed for every $t>0$ as 
a monotone decreasing function 
of the geodesic distance 
as shown by \cite{alonso2021pointwise}. It seems reasonable to expect that similar monotonicity holds for compact symmetric spaces yielding honest diffusion means, for instance, on real and complex projective spaces and on Lie groups. As such, they can be treated in the framework of \Cref{sec:main}.

In~general, diffusion means for fixed values of $t$ can exhibit smeariness and non-uniqueness similarly to intrinsic means. 
However, for increasing $t$, smeariness and non-uniqueness seems less likely, as observed, when diffusion means approach extrinsic means. 
Especially, Theorem 4.8 in \cite{eltzner2023diffusion} shows that simultaneous estimation of the diffusion mean  $\mf_{D,t}(P)$ and the diffusion time $t$ restricts the possibility of full smeariness to a very restricted set of probability measures and allows at most for directional smeariness otherwise. 
Moreover, numerical simulations also suggest that effects of finite sample smeariness for a given data set, i.e., a large variance modulation \eqref{eq:varMod}, are mitigated by additionally optimizing over the time $t$. An 
explanation for this behavior
is given by 
\Cref{thm:minimaxLower} (there $t$ denotes perturbation scale not diffusion time): As the set of probability distributions exhibiting non-unique diffusion means is closely tied to the time parameter $t>0$, 
optimizing over $t$ 
makes it less likely to approach the regime of non-uniqueness. 

While diffusion means have only been recently developed with many fundamental properties still open, the above consideration show that diffusion means may also suffer from deteriorating constants and lower rates, leaving the  precise picture open for future research.

\subsection{Wasserstein Barycenters on Euclidean Spaces}

For the space of Euclidean probability measures $\PC_2(\RR^m)$ with finite second moment 
we use the $2$-Wasserstein distance $\WC$, defined for two probability measures $\mu, \nu\in \PC_2(\RR^m)$ as 
\begin{align*}
    \WC(\mu, \nu) \coloneqq \left(\inf_{\pi\in \Pi(\mu, \qms)} \int \norm{x-y}^2 \dif \pi(x,y)\right)^{1/2},
\end{align*}
where $\Pi(\mu, \qms)$ denotes the collection of probability measures on $\RR^{2m}$ whose marginals on $\RR^m$ coincide with $\mu$ and $\qms$, i.e., 
\begin{align*}
    \Pi(\mu, \qms) \coloneqq \left\{  \pi\in \PC(\RR^{m}\times \RR^m) \;\bigg|\; \begin{array}{c}
         \pi(A\times \RR^m) = \mu(A) \quad \text{ for all } A \in \mathcal{B}(\RR^m)  \\
         \pi(\RR^m \times B) = \qms(B) \quad \text{ for all } B \in \mathcal{B}(\RR^m)
    \end{array}\right\}.
\end{align*}
Jointly, $(\XC, d) \coloneqq (\PC_2(\RR^m), \WC)$ defines the \emph{2-Wasserstein space} over $\RR^m$, see \citet{vil03, villani2008optimal}, \cite{panaretos2020an} for comprehensive monographs.  A notable feature of the Wasserstein metric is that it equips $\PC_2(\RR^m)$ with a geometry that is \emph{consistent} with the Euclidean geometry: For instance, the Wasserstein distance between two multivariate Gaussians of identical covariance coincides with the difference between their expectations.

The Wasserstein barycenter, denoted by $\mf_W$, is the Fr\'echet mean in the Wasserstein space for a probability measure $\pms$. Theoretical foundations on existence and uniqueness were first established by \cite{agueh2011barycenters} and extended by \cite{le2017existence}. Specifically, a Wasserstein barycenter exists if $\pms$ is a Borel measure that admits a finite second moment \citep[Theorem 2]{le2017existence}, i.e., if $\EE_{\mu\sim \pms}[\WC^2(\mu,\delta_{0})]<\infty$, denoted by $\pms\in\PC_2(\PC_2(\RR^m))$. Moreover, if $\pms$ is finitely supported and if one of the support points is a measure that is absolutely continuous with respect to the Lebesgue measure, then the Wasserstein barycenter is unique \citep[Propositions 2.3 and 3.5]{agueh2011barycenters}.  In contrast, if $\pms$ is supported on the collection of measures with finite and discrete support, uniqueness can fail, see \Cref{ex:non-uniqueMeans}.

Given i.i.d.\ random measures $\mu_1, \dots, \mu_n \sim \pms$ with corresponding empirical measure $\ems\coloneqq \frac{1}{n}\sum_{ i= 1}^{n} \delta_{\mu_i}$, \cite{le2022fast} established under certain structural assumptions on $\pms$ parametric convergence rates ($1/n$ for the expectation of the squared distance) of empirical Wasserstein barycenters $\emean \in \mf_{W}(\ems)$ to the respective population Wasserstein barycenter. To formalize their result, the following definition is necessary. 

\begin{definition}
A convex function $\varphi \colon \RR^m \to \RR$ is called \emph{$\alpha$-strongly convex} for $\alpha>0$ if it is differentiable and for every $x,y\in \RR^m$, 
\begin{align*}
    \langle \nabla \varphi(x), x-y\rangle \geq \varphi(x) - \varphi(y) + \frac{\alpha}{2}\norm{x-y}^2.
\end{align*}
Moreover, $\varphi$ is called \emph{$\beta$-smooth} 
for $\beta >0$ if it is differentiable and for every $x,y \in\RR^m$, 
\begin{align*}
    \langle \nabla \varphi(x), x-y\rangle \leq \varphi(x) - \varphi(y) + \frac{\beta}{2}\norm{x-y}^2.
\end{align*}
\end{definition}

These two regularity properties are deeply rooted in the optimization literature since an objective function which fulfills both will exhibit linear convergence rates of gradient descent methods \citep{karimi2016linear}. Notably, if a convex, differentiable function is $\alpha$-strongly convex and $\beta$-smooth, then necessarily $\alpha \leq \beta$. 

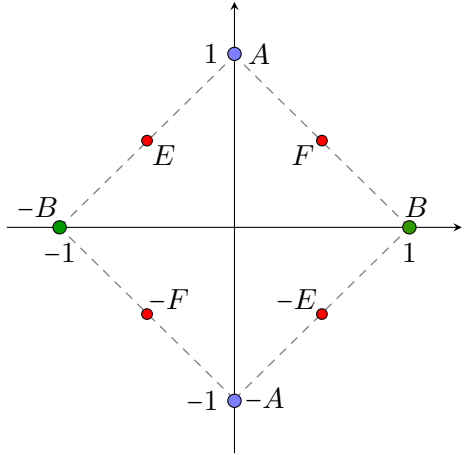
\begin{figure}[b!]
    \centering
\definecolor{ffqqqq}{rgb}{1.,0.,0.}
\definecolor{ttzzqq}{rgb}{0.2,0.6,0.}
\definecolor{qqzzqq}{rgb}{0.,0.6,0.}
\definecolor{xdxdff}{rgb}{0.49019607843137253,0.49019607843137253,1.}

\begin{tikzpicture}[line cap=round,line join=round,>=triangle 45,x=1.0cm,y=1.0cm]
\begin{axis}[
x=2.3cm,y=2.3cm,
axis lines=middle,
xmin=-1.3,
xmax=1.3,
ymin=-1.3,
ymax=1.3,
xtick={-1.0,0,1.0},
ytick={-1.0,0,1.0},]
\clip(-1.3,-1.3) rectangle (1.3,1.3);
\draw [gray, line width=0.5pt,dashed] (-1.,0.)-- (-0,0.9913578179545821);
\draw [gray, line width=0.5pt,dashed] (0.0,-0.9913578179545821)-- (1.,0.);
\draw [gray, line width=0.5pt,dashed] (1.,0.)-- (0,0.9913578179545821);
\draw [gray, line width=0.5pt,dashed] (0.0,-0.9913578179545821)-- (-1.,0.);
\draw (0.02,1.0) node[anchor=west] {$A$};
\draw (0,-0.985) node[anchor=west] {$-A$};
\draw (1.04,0) node[anchor=south] {$B$};
\draw (-0.95,0.0) node[anchor=south east] {$-B$};
\draw (-0.53,0.53) node[anchor=north west] {$E$};
\draw (0.53,-0.53) node[anchor=south east] {$-E$};
\draw (0.515,0.53) node[anchor=north east] {$F$};
\draw (-0.55,-0.53) node[anchor=south west] {$-F$};
\begin{scriptsize}
\draw [fill=xdxdff] (0.,1.) circle (2.5pt);
\draw [fill=xdxdff] (0.,-1) circle (2.5pt);
\draw [fill=qqzzqq] (-1.,0.) circle (2.5pt);
\draw [fill=ttzzqq] (1.,0.) circle (2.5pt);
\draw [fill=ffqqqq] (-0.5,0.5) circle (2.0pt);
\draw [fill=ffqqqq] (0.5,-0.5) circle (2.0pt);
\draw [fill=ffqqqq] (0.5,0.5) circle (2.0pt);
\draw [fill=ffqqqq] (-0.5,-0.5) circle (2.0pt);
\end{scriptsize}
\end{axis}
\end{tikzpicture}
    \caption{Depicting the points from \Cref{ex:non-uniqueMeans}: $E$ (resp. $F$) represents the midpoint between $A$ and $-B$ (resp.\ $A$ and $B$). The set of Wasserstein barycenters for the probability measure $\frac{1}{2}(\delta_{\frac{1}{2}(\delta_{A} + \delta_{-A})} + \delta_{\frac{1}{2}(\delta_{B} + \delta_{-B})})$ is given by $\{\frac{1}{2}(\delta_{E} + \delta_{-E}), \frac{1}{2}(\delta_{F} + \delta_{-F})\}$.}
    \label{fig:barycenter}
\end{figure}

\begin{theorem}[{\citealt[Corollary 16]{le2022fast}}, sufficiency of condition $\beta-\alpha<1$]\label{prop:barycentereConvergence}
    Let $\pms \in \PC_2(\PC_2(\RR^m))$ be a probability distribution with a unique Wasserstein barycenter $\mf_W(\pms) = \{\pmean\}$. Let $0<\alpha\leq \beta$ and suppose that every $\mu \in \supp(\pms)$ is the push-forward of $\pmean$ by the gradient of an $\alpha$-strongly convex and $\beta$-smooth function $\varphi$, i.e., $\mu = (\nabla \varphi)_{\#}\pmean$. Then, if $\beta-\alpha<1$, every empirical Wasserstein barycenter $\emean \in \mf_W(\ems)$ fulfills \begin{align*}
        \EE[ \WC^2(\emean, \pmean)] \leq 
        \frac{1}{n}\, \frac{4 \EE_{\mu \sim \pms}[\WC^2(\mu, \pmean)] }{(\beta - \alpha-1)^2}.
    \end{align*}
\end{theorem}

The result by \cite{le2022fast} leaves open whether the requirement $\beta-\alpha<1$ is a genuine restriction or an artifact of their proof. Indeed, as $\beta-\alpha$ approaches one, the underlying constant blows up to infinity. To analyze necessary conditions on the difference $\beta-\alpha$ to guarantee parametric convergence rates of empirical Wasserstein barycenters we utilize \Cref{cor:effects_FSS} and construct a measure $\pms$ which leads to non-unique means.

\begin{example}[Non-unique Wasserstein barycenter]\label{ex:non-uniqueMeans}
Define the points in the two-dimensional Euclidean plane (see \Cref{fig:barycenter})\begin{align*}
    A = (0, 1)^t\quad\quad  B = (1,0)^t\quad \quad E \coloneqq (-1/2,1/2)^t \quad \quad F \coloneqq (1/2,1/2)^t 
\end{align*}
and define mixtures of Dirac measures,
\begin{align*}
\mu \coloneqq \frac{1}{2}\left(\delta_{A} + \delta_{-A}\right),\quad  \nu\coloneqq \frac{1}{2}\left(\delta_{B} + \delta_{-B}\right),\quad \pms \coloneqq \frac{1}{2}\left(\delta_{\mu} + \delta_{\nu}\right).
\end{align*}
Then, by \citet[Theorem 8]{le2017existence}
the Wasserstein barycenter is given by 
\begin{align*}
    \mf_W(\pms) = \left\{ \frac{1}{2}(\delta_{E} + \delta_{-E}), \frac{1}{2}(\delta_{F} + \delta_{-F}) \right\}.
\end{align*}
\end{example}

Necessary and sufficient conditions for the existence of an $\alpha$-convex and $\beta$-smooth function $\varphi$ with $\beta \geq \alpha > 0$ such that $\xi = \nabla \varphi_{\#}[\frac{1}{2}(\delta_{E} + \delta_{-E})]$ for $\xi \in \{\mu, \nu, \frac{1}{2}(\delta_{E} + \delta_{-E})\}$ were derived by \citet[Theorem 4]{taylor2017smooth}. Invoking their analysis and utilizing \Cref{cor:effects_FSS}, we arrive at the following proposition.

\begin{proposition}[Necessity of condition $\beta-\alpha \leq 2$ for fast rates]\label{lem:lowerbound_difference}
    In the setting of \Cref{ex:non-uniqueMeans}, the following assertions hold. 
    \begin{enumerate}
        \item[$(i)$] Every $\alpha$-strongly convex and $\beta$-smooth function $\varphi \colon \RR^2 \to \RR$ with $\beta > \alpha \geq 0$ such that $\mu = \nabla \varphi_{\#}[\frac{1}{2}(\delta_{E} + \delta_{-E})]$ fulfills $\beta-\alpha \geq 2.$ Further, if $\alpha$ is positive, then $\beta-\alpha >2$.
        \item[$(ii)$] For every $\delta>0$ there exists a choice for $\beta \geq \alpha >0$ with $\beta-\alpha \leq 2+ \delta$ such that for every $\xi \in \{\mu, \nu, \frac{1}{2}(\delta_{E} + \delta_{-E})\}$ there exists an $\alpha$-strongly convex and $\beta$-smooth function $\varphi$ which fulfills  $\xi = \varphi_{\#}[\frac{1}{2}(\delta_{E} + \delta_{-E})]$.
        \item[$(iii)$] For every estimator $\hat {\pms}\colon \bigcup_{n \in \NN} \PC_2(\RR^m)^n \to  \PC_2(\RR^m)$ and every sample $\mu_1, \dots, \mu_n$ of size $n\in\NN$ there exists a probability measure $P_{\hat {\pms},n}\in \{(1-t)\pms + t \delta_\xi \;|\; \xi  \in \mf(\pms), t \in (0,1)\}$ with a unique Wasserstein barycenter $\mf_W(P_{\hat {\pms},n})=\{\pmean_{\hat {\pms},n}\}$ such that 
        \begin{align*}
                   \mathbb{E}\left[ \WC^2(\hat  {\pms}(\mu_1, \dots, \mu_n),\pmean_{\hat {\pms},n})\right] \geq \diam{\mf_W(\pms)}^2/9 = 1/9. 
        \end{align*}
    \end{enumerate}
    \end{proposition}

The above proposition highlights that the upper bound on the difference $\beta-\alpha<2$ is \emph{necessary} to infer parametric convergence rates for the empirical Wasserstein barycenter. In particular, 
this degenerate behavior cannot be overcome by whatever estimator in the regime $\beta -\alpha> 2$ without imposing additional assumptions on the underlying measures. 

It remains open to investigate whether the analysis by \cite{le2022fast} is optimal or whether it can be extended to the regime  $1\leq \beta-\alpha < 2$. In the introduction of their paper, they state that the bounds for the support of probability measures on spheres $\SSS^m$ amounts to an open ball of radius $\frac{\pi}{4}$, which is  by a factor of $2$ short of the bound of $\frac{\pi}{2}$ given by \cite{afsari2011riemannian}. For this reason, we conjecture that their result extends to the regime $1\leq \beta-\alpha < 2$.

\begin{conjecture}
With the notation of Theorem \ref{prop:barycentereConvergence},
 \begin{align*} 
    \EE[ \WC^2(\emean, \pmean)] \leq 
    \frac{C}{n}\, \frac{ \EE_{\mu \sim \pms}[\WC^2(\mu, \pmean)] }{(\beta - \alpha-2)^2},
\end{align*} for all $1\leq \beta-\alpha < 2$ with a constant $C>0$
that does not depend on $\pms$. 
\end{conjecture}

In addition, it would be worthwhile to explore sharp convergence rates for the empirical Wasserstein barycenter beyond this specific range. A notable contribution in this context was made by \citet[Theorem 4.6.1]{kroshnin2021inside} who provides an example where $\pms$ is concentrated on the collection of finitely supported probability measures and the empirical barycenters achieves a logarithmic convergence rate at best.

\section{Discussion}\label{scn:discussion}

This work sheds light on the statistical challenges involving (generalized) Fr\'echet means nearby the regime of non-uniqueness, demonstrating that estimating Fr\'echet means with high accuracy is not possible (\Cref{thm:minimaxLower}) and that all testing procedures will experience large errors of first or second kind (\Cref{rmk:testing}). For practical contexts this raises two points. Firstly, how to identify the presence of non-uniqueness effects, and secondly, how to proceed when such effects cannot be dismissed.

When the objective is to determine a location descriptor (a Fr\'echet $\rho$-mean for given~$\rho$) for a data set, nearby the regime of non-unique means, it is reasonable to anticipate that the distribution of an appropriate estimator (such as the sample mean) based on i.i.d.\ observations is concentrated around the respective population means. As the sample size increases and assuming the estimator is consistent, this distribution tends to become multimodal, a characteristic that can be assessed using the test by \cite{eltzner2020testing}. If this test does not reject nonuniqueness, a single location descriptor likely is not sufficient to suitably describe the data set. Instead, it is reasonable to perform cluster analysis to partition the data into structured sub-samples. For each sub-sample the Fr\'echet mean is then to be calculated. These Fr\'echet means then provide a more accurate representation of location descriptors. 

Moreover, when employing Fr\'echet means for testing purposes between datasets, unless non-uniqueness can be ruled out, we suggest conducting a sample splitting procedure to assess if the respective test maintains the correct significance level between split samples from both datasets. If the effective level strongly deviates from the nominal level, then, either recalibrate, for instance by bootstrapping, or consider an alternative testing method. The latter, however, amounts to multiple testing bias, warranting appropriate correction methods. We leave a comprehensive analysis of such procedures as future research.

\begin{ackno}
    S.\,Hundrieser gratefully acknowledges support by the Deutsche Forschungsgemeinschaft (DFG, German Research Foundation) as part of DFG RTG 2088, B.~Eltzner and S.F.\,Huckemann gratefully acknowledge funding by the DFG SFB 1456. The latter also acknowledges support by DFG HU 1575/7.
    \end{ackno}


\appendix

\section{Deferred Proofs}

\begin{proof}[Proof of \Cref{prop:CLT_Circle}]
    Assertion $(i)$ has been shown by \cite{mckilliam2012direction,hotz2015intrinsic}. For Assertion $(ii)$,  expand the empirical Fr\'echet function based on the arc length distance for $0<x<\delta$, 
    \begin{align*}
     F_n(x) &= \frac{1}{n}\sum_{X_j \in [x -\pi,\pi)} (X_j -x)^2 + \frac{1}{n}\sum_{X_j <x-\pi} (X_j+2\pi - x)^2 \\
           &= \frac{1}{n}\sum_{j=1}^n(X_j - x)^2 +  \frac{4\pi}{n} \sum_{X_j < x-\pi} (X_j - x + \pi)
    \end{align*}
    In consequence, with the Euclidean mean $\overline{X}_n=\frac{1}{n}\sum_{j=1} X_j$, almost surely,
    \begin{align}\label{eq:proof-clt}
     \frac{1}{2}\,\grad\, F_n(x)  &= x-\overline{X}_n -  \frac{2\pi}{n} \sum_{X_j < x-\pi} 1\,.
   \end{align}
    Further, upon rewriting the expectation,  
    \begin{align*}
    \EE\left[\frac{1}{n} \sum_{X_j < x-\pi} 1\right] &= \int_{-\pi}^{x-\pi} f(t)\,dt~=~
     \int_{-\pi}^{x-\pi} \left(\frac{1}{2\pi} - G'(\pi +t) + o(G'(\pi+t))\right) \,dt\\
     &= \frac{x}{2\pi} - G(x) + o\big(G(x)x\big)\,,
    \end{align*}
    where the last equality is due to the properties of $G$, we infer for the sum of Bernoulli variables
    \begin{align*}
    \frac{1}{n} \sum_{j=1}^n \Indicator{\{X_j < x-\pi\}} &= \frac{x}{2\pi} - G(x) + o\big(G(x)x\big) + O_p\left(\frac{\sqrt{|x-G(x)|}}{\sqrt{n}}\right)\,.
    \end{align*}
    In conjunction with the analogue argument for $-\delta < x < 0$ we thus obtain from (\ref{eq:proof-clt}) that
    \begin{align}\label{proof1:prop:CLT_Circle}
    \frac{\sqrt{n}}{2}\, \grad\, F_n(x)  &= \sqrt{n} (\sign(x)\,2\pi\,G(|x|) -\overline{X}_n)  + o\big(G(x)x\big) + O_p\left(\frac{\sqrt{|x+G(|x|)|}}{\sqrt{n}}\right)\,.
   \end{align}
    This yields Assertion $(ii)$ since the above left-hand side\ vanishes by definition at $x=\emean$, it holds $\emean = o_p(1)$ due to \cite{ziezold1977expected}, and by the classical Euclidean central limit theorem which asserts $\sqrt{n}\,\overline{X}_n  \xrightarrow{\mathcal{D}} \mathcal{N}\left(0,\sigma^2 \right)$.
   \end{proof}

\begin{proof}[Proof of \Cref{cor:log-smeary}]
    Since the density of $\pms_{\textup{log},r}$ on $(-\pi, -\pi+1/2)\cup (\pi - 1/2, \pi)$ is strictly smaller than $\frac{1}{2\pi}$ we see that $c_r>\frac{1}{2\pi}$. Hence, by \citet[Propostion 1]{hotz2015intrinsic} there is a unique Fr\'echet population mean $\pms$ in $[-1/2,1/2]$, which, by symmetry, is $\mf_I= \{0\}$. This implies the preliminary claim. Moreover, the strong law by \cite{ziezold1977expected} yields $(i)$.

    Next, plug in $x=\emean$ in (\ref{proof1:prop:CLT_Circle}) as done in the proof of Theorem \ref{prop:CLT_Circle}, to obtain
    $$ \sqrt{n}\cdot \text{sign}(\emean )\exp(-1/|\emean |^r)  = \frac{\sqrt{n}}{2\pi}\,\overline{X}_n +o_P(1)\,.$$
    Subjecting both sides above to $y\mapsto \sign{(y)}\cdot \log(|y|) \,\Indicator{\RR\setminus\{0\}}(y)$ gives
    \begin{align*} \text{sign}(\emean ) \left(\log \sqrt{n} - \frac{1}{|\emean|^r}\right) &= \text{sign}(\overline{X}_n) \log\left| \frac{\sqrt{n}}{2\pi}\,\overline{X}_n + o_P(1)\right|\\
    &= \text{sign}(\overline{X}_n) \log\left| \frac{\sqrt{n}}{2\pi}\,\overline{X}_n\right| + o_P(1)\,, 
    \end{align*}
    where for the second equality we used that  $\frac{2\pi}{\sqrt{n}\,\overline{X}_n} = O_P(1)$ and $\log\big(|1+ o_P(1)|\big) = o_P(1)$. Setting $L_n = \text{sign}(\overline{X}_n) \log\left| \frac{\sqrt{n}}{2\pi}\,\overline{X}_n\right|$, we proceed to 
    \begin{align*} \text{sign}(\emean ) \,\log \sqrt{n} - L_n &=  \frac{\text{sign}(\emean ) + o_p(|\widehat\pms|^r)}{|\emean|^r}\,.
    \end{align*}
    This yields whenever $|L_n| < \log \sqrt{n}$, and the probability for this event tends to $1$, that
    \begin{align*}  \text{sign}(\emean)\,|\emean|^r &=  \frac{1 + o_p(|\emean|^r)}{\text{sign}(\emean ) \,\log \sqrt{n} - L_n}
    ~=~
    \frac{1 + o_p(|\emean|^r)}{\text{sign}(\emean )\,\log\sqrt{n} }\, \sum_{k=0}^\infty \left(\frac{\text{sign}(\emean )\,L_n}{\log\sqrt{n}}\right)^k\,.
    \end{align*}
    Hence, since $\emean = o_P(1)$ due to $(i)$ we conclude
    \begin{equation*}
         |\emean|^r	= \frac{1}{\log\sqrt{n}} +  \frac{\sign(\emean)\,L_{n}}{(\log\sqrt{n})^2} + o_P\left(\frac{1}{(\log\sqrt{n})^2}\right)\,.
    \end{equation*}
    By symmetry of $\pms_{\textup{log},r}$, this yields Assertion $(ii)$. It further implies 
    \begin{align*}
        \emean &= \sign(\emean)\left(\frac{1}{\log\sqrt{n}} +  \frac{\sign(\emean)\,L_{n}}{(\log\sqrt{n})^2} + o_P\left(\frac{1}{(\log\sqrt{n})^2}\right)\right)^{\frac{1}{r}}\\
        &=\frac{\sign(\emean)}{(\log\sqrt{n})^\frac{1}{r}} \left(1 + \frac{1}{r}\, \frac{\sign(\emean)\,L_{n}}{\log\sqrt{n}} + o_P\left(\frac{1}{\log\sqrt{n}}\right)\right)  \,.
    \end{align*}
    With the definition of $L_n$ above, we thus obtain Assertion $(iii)$.
\end{proof}

\begin{proof}[Proof of \Cref{prop:Procrustes}]
With the notation of  \Cref{prop:Procrustes},
let $\hat v_1 := {\rm vec}(\psi_{\xi}\circ \bar P_n)$, $X_i : = {\rm vec}(\psi_{\xi_n}\circ Z_i)$, $i=1,\ldots,n$ and $X:= {\rm vec}(\psi_\xi\circ Z)$. Further, let $D_n:=\frac{1}{n}\sum_{i=1}^n X_iX_i^T = \sum_{j=1}^{N+1} \hat \lambda_j \hat v_j \hat v_j^T$, with $\hat\lambda_1\geq \ldots \geq \hat\lambda_{N+1} $ and $\hat v_1,\ldots,\hat v_{N+1}$ such that $v_1^T\hat v_1 \geq 0$. In fact, by the strong law of \cite[Theorem A.3]{huckemann2011intrinsic} we may assume that $v_1^T\hat v_1 > 0$ for $n$ sufficiently large a.s. and thus $\hat y := \phi_{v_1}(\hat v_1) \in T_{v_1}\SSS^N$ is well defined. 
Since the empirical Fr\'echet $d^2_R$-function, in the local coordinate $y \in T_{v_1}\SSS^N$ is given by 
\begin{align*} F_n\left({\rm vec}^{-1}\circ \phi_{v_1}^{-1}(y)\right)  =& 1 - \frac{(y+v_1)^T}{\|y+v_1\|}D_n \frac{y+v_1}{\|y+v_1\|}~=:~G_n(y)\,,
\end{align*}   
we have, a.s. for $n$ sufficiently large and suitable $\widetilde{y}$ between $\hat y$ and $0$ that
\begin{align*} 
0 =& \grad G_n(\hat y) = \grad G_n(0) + {\rm Hess} G_n(\widetilde y) \hat y\,.
\end{align*}
As shown in \cite{HE_MFO_2019}, the central limit theorem of \cite{eltzer2019_smearyCLT} is applicable, yielding the distributional limit
$$ \sqrt{n}\hat y \to {\mathcal N}\big(0, ({\rm Hess}_y G(0))^{-} D ({\rm Hess}_y G(0))^{-}\big)\,,$$
where $D= \EE[XX^T]$, $G(y) = 1 -  \frac{(y+v_1)^T}{\|y+v_1\|}D\frac{y+v_1}{\|y+v_1\|}$ is the population Fr\'echet function in the local coordinate and $({\rm Hess}_y G(0))^{-}$ is any generalized inverse of ${\rm Hess}_y G(0)$ (justified below). Since 
\begin{align*} 
-\frac{1}{2}\grad_y G(y) =& D\frac{v_1+y}{\|y+v_1\|^2} - \frac{(y+v_1)^T D (y+v_1)}{\|y+v_1\|^4}(y+v_1)\\
-\frac{1}{2}{\rm Hess}_y G(y) =& \frac{D}{\|y+v_1\|^2} - 2\frac{ D (y+v_1)(y+v_1)^T + (y+v_1)(y+v_1)^T D}{\|y+v_1\|^4} \\ & + 4 \frac{(y+v_1)^T D (y+v_1)}{\|y+v_1\|^6} (y+v_1)(y+v_1)^T - \frac{(y+v_1)^T D (y+v_1)\Id_{m+1}}{\|y+v_1\|^4}
\end{align*}   
we have thus
$$ {\rm Hess}_yG(0) = 2\sum_{j=2}^{N+1}(\lambda_1 - \lambda_j) v_jv_j^T\,.$$
Since, moreover a.s., for $n$ sufficiently large, $\hat y$ is in the tangent space spanned by $v_2,\ldots,v_{N+1}$, any generalized inverse can be taken, yielding the assertion.
\end{proof}

\begin{proof}[Proof of \Cref{lem:lowerbound_difference}]
    We prove Assertions $(i)$ and $(ii)$ by invoking \citet[Theorem 4 and Corollary 1]{taylor2017smooth}, which states that a function $\tilde \varphi \colon \{E, -E\} \to \RR$ can be extended to an $\alpha$-strongly convex and $\beta$-smooth function $\varphi \colon \RR^2 \to \RR$ for $\beta > \alpha \geq 0$ and such that $\nabla \varphi(E) = A$ and $\nabla \varphi(-E) = -A$ if and only if 
    \begin{align*}
        \tilde \varphi(E) - \tilde \varphi(-E) &- (-A)^t(E +E) \\ 
        \geq \;& \frac{1}{2(1-\alpha/\beta)}\left(\frac{1}{\beta} \norm{2A}^2 + \alpha \norm{2 E}^2 - 2 \frac{\alpha}{\beta}(-2A)^t(-2E) \right), \text{ and } \\
        \varphi(-E) - \varphi(E) &- (A)^t(-E -E) \\ 
        \geq \;& \frac{1}{2(1-\alpha/\beta)}\left(\frac{1}{\beta} \norm{2A}^2 + \alpha \norm{2 E}^2 - 2 \frac{\alpha}{\beta}(2A)^t(2E) \right).
    \end{align*}
Since the right-hand sides are  identical for both inequalities, we may assume without loss of generality that $  \tilde \varphi(E) - \tilde \varphi(-E)\eqqcolon \kappa \leq 0$ and only continue with the first inequality. Computing the terms on both sides then yields that
\begin{align*}
  1 \geq   \kappa  +1 \geq \;& \frac{\beta }{2(\beta - \alpha)}\left(\frac{1}{\beta} \norm{2A}^2 + \alpha \norm{2 E}^2 - 2 \frac{\alpha}{\beta}(-2A)^t(-2E) \right)\\
    =\; & \frac{1}{2(\beta - \alpha)}\left(\norm{2A}^2 + \alpha\beta \norm{2 E}^2 - 2 \alpha(-2A)^t(-2E) \right)\\
    =\; &  \frac{1}{2(\beta - \alpha)} \left(4 + 2\alpha \beta -4\alpha  \right)
    = \frac{1}{2(\beta - \alpha)} \left( 4 + 2\alpha (\beta -2)\right).
\end{align*}
Performing a change of variables $\gamma = \beta - \alpha \geq 0$ thus yields the following inequality 
\begin{align*}
    \gamma \geq 2 + (\beta - \gamma)(\beta -2)
\end{align*}
Solving this inequality via the quadratic formula, we obtain the set of solutions 
\begin{align}\label{eq:GammaBeta}
   \left\{(\gamma, \beta)\in \RR_{\geq 0}^2 \;\bigg|\; \frac{1}{2} \left(-\sqrt{\gamma^2 - 4} + \gamma + 2\right)\leq \beta\leq \frac{1}{2} \left(\sqrt{\gamma^2 - 4} + \gamma + 2\right),  \gamma\geq 2\right\}.
\end{align}
The minimal value for $\gamma$ is given by $2$, and it leads to $\beta = 2$, which implies that $\alpha =0$. Consequently, if $\alpha>0$, then $\beta-\alpha=\gamma>2$. 

Meanwhile, imposing the assumption $\nabla \varphi(E) = -A$ and $\nabla \varphi(-E) = A$ would lead to 
 \begin{align*}
       \tilde  \varphi(E) - \tilde \varphi(-E) &- (A)^t(E +E) \\ 
        \geq \;& \frac{1}{2(1-\alpha/\beta)}\left(\frac{1}{\beta} \norm{-2A}^2 + \alpha \norm{2 E}^2 - 2 \frac{\alpha}{\beta}(2A)^t(-2E) \right),
\end{align*}
which is equivalent to 
\begin{align*}
  -1 \geq   \kappa  -1 \geq\frac{1}{2(\beta - \alpha)} \left(4 + 2\alpha \beta +4\alpha  \right),
\end{align*}
and thus has no solutions for $\beta > \alpha \geq 0$. From these considerations we infer Assertion~$(i)$. 

To show Assertion $(ii)$ we first choose $\gamma= 2+ \max(\delta,1/2)$. Then, infer by \eqref{eq:GammaBeta} there exists an $\alpha$-strongly convex and $\beta$-smooth function  $\varphi^\mu\colon \RR^m \to \RR$ such that $\varphi^\mu(E) = \varphi^\mu(-E)$ with $\beta =  (-\sqrt{\gamma^2 - 4} + \gamma + 2)/2\leq 3$ and $\alpha = \beta - \gamma>0$, and which fulfills $\mu = \nabla \varphi^\mu_{\#}[\frac{1}{2}(\delta_{E} + \delta_{-E})]$.
Likewise, since $A^t E = B^t E$, there also exists an $\alpha$-strongly convex and $\beta$-smooth function $\varphi^\nu\colon \RR^m \to \RR$ which fulfills $\nu = \nabla \varphi^\nu_{\#}[\frac{1}{2}(\delta_{E} + \delta_{-E})]$. Finally, for the remaining case,  we consider the function  $\varphi^I\colon \RR^m \to \RR, x\mapsto \frac{1}{2}\norm{x}^2$ which satisfies that $\frac{1}{2}(\delta_{E} + \delta_{-E}) = \nabla \varphi^I_{\#}[\frac{1}{2}(\delta_{E} + \delta_{-E})]$. In particular, $\varphi^I$ is $1$-strongly convex and $1$-smooth, and since $\beta\leq 3$ and $\alpha< 1$, the validity of Assertion $(ii)$ follows. 

Finally, Assertion $(iii)$ follows from \Cref{cor:effects_FSS} after  noting that $$\diam{\mf_W(\pms)}^2 = \WC^2\left(\frac{1}{2}(\delta_{E} + \delta_{-E}), \frac{1}{2}(\delta_{F} + \delta_{-F})\right)= 1$$ since $\norm{E - F}^2 = \norm{E + F}^2 = 1$. 
\end{proof}

\end{document}